\documentclass[12pt, a4paper, leqno]{amsart}
\usepackage[utf8]{inputenc}
\usepackage{amsmath}
\usepackage{amsfonts}
\usepackage{amssymb}
\usepackage{times}
\usepackage[T1]{fontenc} 
\usepackage{url}
\usepackage[dvipsnames]{xcolor}
\usepackage[colorlinks, allcolors=RedViolet]{hyperref} 
\usepackage{graphicx}
\usepackage{esint}
\hypersetup{pdfstartview=}

\let\oldmarginpar\marginpar
\renewcommand\marginpar[1]{\-\oldmarginpar[\raggedleft\footnotesize #1]%
{\raggedright\footnotesize #1}}

\usepackage{amsthm}
\theoremstyle{plain}
\newtheorem{thm}[equation]{Theorem}
\newtheorem{lem}[equation]{Lemma}

\newtheorem{prop}[equation]{Proposition}
\newtheorem{cor}[equation]{Corollary}

\theoremstyle{definition}

\theoremstyle{remark}

\newtheorem{remark}[equation]{Remark}

\numberwithin{equation}{section}

\newcommand{\R}{\mathbb{R}}
\newcommand{\N}{\mathbb{N}}

\newcommand{\Rn}{{\mathbb{R}^n}}

\newcommand{\D}{{\mathcal D}}
\renewcommand{\H}{{\mathcal H}}

\renewcommand{\div}{\divop}
\renewcommand{\phi}{\varphi}
\renewcommand{\rho}{\varrho}
\renewcommand{\epsilon}{\varepsilon}
\renewcommand{\vartheta}{\theta}

\def\le{\leqslant}

\def\ge{\geqslant}

\def\diam{\qopname\relax o{diam}}

\def\div{\qopname\relax o{div}}

\def\dist{\qopname\relax o{dist}}

\def\I{\operatorname{\mathcal{I}}}
\def\J{\operatorname{\mathcal{J}}}

\newcommand{\ainc}[1]{\hyperref[defn:aInc]{{\normalfont(aInc){\ensuremath{_{#1}}}}}}
\newcommand{\adec}[1]{\hyperref[defn:aDec]{{\normalfont(aDec){\ensuremath{_{#1}}}}}}
\newcommand{\adeci}[1]{\hyperref[defn:aDeci]{{\normalfont(aDec){\ensuremath{_{#1}^\infty}}}}}

\date{\today}

\usepackage{color}
\definecolor{blau}{rgb}{0.1,0.0,0.9}

\newcounter{komcounter}
\numberwithin{komcounter}{section}

\begin{document}

\title{Double phase image restoration}


\author{Petteri Harjulehto}
 \address{Petteri Harjulehto,
 Department of Mathematics and Statistics,
FI-20014 University of Turku, Finland}
\email{\texttt{petteri.harjulehto@utu.fi}}

\author{Peter Hästö}
 \address{Peter Hästö, Department of Mathematics and Statistics,
FI-20014 University of Turku, Finland, and Department of Mathematics,
FI-90014 University of Oulu, Finland}
\email{\texttt{peter.hasto@oulu.fi}}





\begin{abstract}
In this paper we explore the potential of the double phase functional in 
an image processing context. To this end, we study minimizers of the double phase energy 
for functions with bounded variation and show that this energy can be obtained 
by $\Gamma$-convergence or relaxation of regularized functionals. 
A central tool is a capped fractional maximal function of the 
derivative of $BV$ functions. 
\end{abstract}

\maketitle


\section{Introduction}


The double phase functional was introduced in the 1980s by Zhikov \cite{Zhi86}, but has 
only recently become the focus of intense research, starting in 2015 with 
Baroni, Colombo and Mingione \cite{BarCM15, BarCM18, ColM15a, ColM16}. 
Subsequently, many other researchers studied double phase problems as well, see, e.g., 
\cite{ByuO17,  CupGGP18, DeFOh19,  DeFM_pp2, Ok18, Ok_pp} for regularity theory, 
\cite{ByuRS18, DeFM_pp1, Shi_pp} for Calder\'on--Zygmund estimates and 
\cite{GeLL19, LiuD18, MaeOMS19} for some other topics. 
Generalizations of the double phase functional have been studies, e.g.\ in
\cite{BenK_pp, EleMM16b, EleMM18, GwiSZ18, HarHT17, HasO_pp19, OhnS18, WanLZ19}.

Zhikov's original motivation for his functionals with non-standard growth was 
modelling physical phenomena. Another of his models, the variable exponent functional, 
was later applied also to the context of image processing, 
see \cite{AlaNA14, CheLR06, HarHLT13, LiLP10}. 
In this article, we demonstrate the potential also of the double phase functional in 
the image processing domain. This is to the best of our knowledge the first 
paper to consider the double phase functional in the space $BV$ of functions 
of bounded variation. 

In mathematical image processing, we interpret a function $u:\Omega\to \R$ as 
the gray-scale intensity at each location. If the function is discretized, we obtain 
an array of pixels common in computer implementations. Typically, $\Omega$ is a rectangle 
and the image contains different objects whose edges correspond to discontinuities 
of $u$. The presence of discontinuities makes this field challenging to approach 
with tools of analysis, but the $BV$ space has proven 
useful. We refer to the book \cite{AubK06} by Aubert and Kornprobst for an overview of 
PDE-based image processing. 

The classical ROF-model \cite{RudOF92} for image restoration calls for minimizing the energy 
\[
\inf_u \int_\Omega |\nabla u| + |u-f|^2\, dx,
\]
where $f$ is the given, corrupted input image that is to be restored. Here $|u-f|^2$ is a
fidelity term which forces $u$ to be close to $f$ on average, whereas the regularizing 
term $|\nabla u|$ limits the variation of $u$. This model is known to be prone 
to a stair-casing or banding effect whereby piecewise constant minimizers are often 
produced \cite{ChaL97}. On the other hand, replacing $|\nabla u|$ by $|\nabla u|^2$ 
leads to a heat-equation type problem, and solutions which are $C^\infty$. This is 
not usually desirable in the image processing context, as edges become blurred. 

The energy of the double phase functional combines growth with two different powers. 
It is given by the expression 
\[
\int_\Omega |\nabla u|^p + a(x) |\nabla u|^q\, dx.
\]
Here $a\ge 0$ is a bounded function and $p<q$. All the previously mentioned 
double-phase references concern super-linear growth (usually $p>1$, but see also \cite{DeFM_pp2}). 
However, for image processing, 
the case $p=1$ and $q=2$ is especially interesting (see above and the discussion in \cite{CheLR06}). 
Then the first term corresponds to the ROF-model, whereas the second term introduces a smoothing effect 
when $a>0$. The parameter $a$ is chosen such that $a=0$ at the edges in 
the image and $a>0$ elsewhere. Usually, the location of the edges is not known, so 
in applications $a$ is estimated from the initial data $f$. Then this adaptive model 
can avoid the stair-casing effect of the ROF-model. 

In the case $p=1$, the double phase energy must naturally be studied in a space 
of $BV$-type. It is not difficult to prove existence of the minimizer even in this 
case (cf.\ Proposition~\ref{prop:existence}). However, the $BV$-space is quite ill-behaved, 
so it is useful for practical implementations to approximate the energy by more 
regular functionals (see, e.g., \cite[Section~6]{Ves01} in the image processing context). The notion of 
$\Gamma$-convergence is often employed in this situation \cite{Bra02, Dal93}, and this article is no 
exception: our main result (Theorems~\ref{thm:Gamma} and \ref{thm:Gamma2}) shows that the $BV$ double phase functional 
(with fidelity term)
\[
|Du|(\Omega) +  \int_\Omega (a(x) |\nabla u|)^2 + |u-f|^2  \, dx
\]
can be approximated in the sense of $\Gamma$-convergence by both
\[
\int_\Omega |\nabla u|^{1+\epsilon} + (a(x) |\nabla u|)^2 + |u-f|^2  \, dx
\quad\text{and}\quad
\int_\Omega |\nabla u|+ (\epsilon + a(x)^2) |\nabla u|^2 + |u-f|^2  \, dx.
\]
Finally, in Corollary~\ref{cor:relaxation}, we show that the $BV$ double phase functional 
can be understood as the relaxation of the $W^{1,1}$ double phase functional. 

Note that we use $a$ inside the power-function, $(a(x)t)^2$. This is of course equivalent 
to having another function outside, but it turns out that the condition on $a$ can be more conveniently 
expressed with this formulation (see Remark~\ref{rem:condition}). 


\section{Notation and existence of minimizers of bounded variation}

We consider subsets of the Euclidean space $\Rn$, $n\ge 2$. The most interesting 
case for image processing is $n=2$, but we can include higher dimensions without
extra complication. 
By $\Omega \subset \Rn$ we denote a bounded domain, i.e.\ an open and connected 
set. The notation $f\lesssim g$ means that there exists a constant
$C>0$ such that $f\le C g$. 
By $c$ we denote a generic constant whose value may change between appearances.
Let $a \in L^\infty (\Omega)$ be non-negative.
By $L^p_a(\Omega)$ we denote the weighted Lebesgue space with weight $a$, given by 
the norm
\[
\|u\|_{L^p_a(\Omega)} := \|a\,u \|_{L^p(\Omega)} = \Big(\int_\Omega (a(x)|u|)^p \, dx\Big)^\frac1p.
\]
$W^{1,p}_a(\Omega)$ is the corresponding Sobolev space. Note that we use the ``weight as multiplier'' formulation, 
so the corresponding weighted measure is $d\mu = a^p\, dx$, not $d\mu = a\, dx$.
By $\H^k$ we denote the $k$-dimensional Hausdorff measure. 
By $|\mu|$ we denote the total variation measure of a vector measure $\mu$, defined as
\[
|\mu|(A) = \sup\Big\{ \sum_{i\in \N} |\mu(A_i)|\, \Big|\, \bigcup_{i\in \N} A_i=A,\ A_i \text{ disjoint and measurable}\Big\}.
\]
By $Mu$ we denote the Hardy--Littlewood maximal function of $u$. 

A function $u \in L^1(\Omega)$ has \textit{bounded variation}, denoted $u \in BV(\Omega)$, if
\[
|Du|(\Omega) := \sup \bigg\{ \int_\Omega u \div \phi \, dx \, \Big|\, \phi\in C^1_0(\Omega;\Rn), 
|\phi|\le 1 \bigg\} < \infty.
\]
Note that this quantity is sometimes denoted by $\|Du\|(\Omega)$. We follow the notation of
\cite{AmbFP00}, which is convenient since it turns out that $|Du|$ is the total variation of a vector measure $Du$. Furthermore, $Du$ can be decomposed as  
\begin{equation}\label{eq:decomposition}
Du = \nabla u \, \H^n + (u_+-u_-) \nu_u \H^{n-1}|_{J_u} + C_u,
\end{equation}
where $\nabla u$ is the absolutely continuous part of the derivative, 
$u_+-u_-$ is the essential point-wise jump of the function, 
$\nu_u$ is the normal of the level-set, $J_u$ is a set of Hausdorff dimension at most
$n-1$ \cite[Theorem~2.3]{Amb90} and the Cantor part $C_u$ has the property that 
$C_u(A)=0$ if $\H^{n-1}(A)<\infty$ \cite[Proposition~3.92]{AmbFP00}. 
The space $BV$ has the following precompactness property \cite[Proposition~3.13]{AmbFP00}: 
if $\sup_i \big(|Du_i|(\Omega) + \|u_i\|_{L^1(\Omega)}\big)<\infty$, then 
there exists a subsequence, denoted again by $(u_i)$, and $u\in BV(\Omega)$ such that 
\begin{equation}\label{eq:precomactness}
u_i\to u\text{ in } L^1(\Omega)
\quad\text{and}\quad
|Du|(\Omega)\le \liminf |Du_i|(\Omega).
\end{equation}
The derivative of the convolution of a $BV$-function can be calculated as expected 
using either the derivative-measure or the function \cite[Proposition~3.2 and equation (2.2)]{AmbFP00}:
\begin{equation}\label{eq:convolution}
\nabla(u*\eta_\delta)(x) 
= \int_\Rn  \eta_\delta(x-y)\, dDu(y) 
= \int_\Rn u(y) \nabla \eta_\delta(x-y)\, dy.
\end{equation}
We refer to \cite{Amb90, AmbFP00,Bra02} for more information about $BV$ spaces.

We abbreviate 
$
BV^{1,2}_a(\Omega) := BV(\Omega)\cap W^{1,2}_a(\Omega)\cap L^2(\Omega)
$
and define for $u \in BV^{1,2}_a(\Omega)$ and initial data $f \in L^2(\Omega)$ 
the $BV$ double phase functional 
\[
\I(u, A):=  |Du|(A) +  \int_A (a(x) |\nabla u|)^2 + |u-f|^2  \, dx
\]
for measurable $A\subset\Omega$. 
We can easily show the existence of a minimizer for this functional using the direct method
of calculus of variations:

\begin{prop}\label{prop:existence}
There exists a unique minimizer $u\in BV^{1,2}_a(\Omega)$, i.e.
\[
\I(u,\Omega) = \inf_{v\in BV^{1,2}_a(\Omega)} \I(v,\Omega).
\]
\end{prop}
\begin{proof}
Let $u_i$ be a minimizing sequence, that is $u_i\in  BV^{1,2}_a(\Omega)$ with 
\[
\lim_{i\to\infty}\I(u_i,\Omega) = \inf_{v\in BV^{1,2}_a(\Omega)} \I(v,\Omega).
\]
By $BV$-precompactness \eqref{eq:precomactness} 
there exists a subsequence, denoted again by $(u_i)$, such that $u_i\to u$ in $L^1(\Omega)$ and 
$|Du|(\Omega) \le \liminf |Du_i|(\Omega)$.
The space $W^{1,2}_a(\Omega)$ is reflexive \cite[Theorem~3.6.8]{HarH19}, so we can find a 
weakly convergent subsequence $(u_i)$. By \cite[Theorem~2.2.8]{DieHHR11}, the modular in 
$W^{1,2}_a(\Omega)$ is weakly lower semicontinuous, so that
\[
\int_\Omega (a(x) |\nabla u|)^2 \, dx \le \liminf \int_\Omega (a(x) |\nabla u_i|)^2 \, dx.
\]
The inequality for the term $|u-f|^2$ follows analogously. Hence $u$ is a minimizer. 

Finally, we note that the $BV$ and $W^{1,2}_a$ parts are convex and 
the $|u-f|^2$ part is strictly convex, so the usual argument yields uniqueness, namely, 
if $u$ and $v$ are distinct minimizers, then we obtain a contradiction from 
$\I(\frac{u+v}2,\Omega)< \frac12 (\I(u,\Omega)+ \I(v,\Omega))$. 
\end{proof}


\section{Lower estimates for the \texorpdfstring{$BV$}{BV} double phase functional}

To be able to construct the minimizers of $\I$ with some numerical scheme, we must show that 
the $BV$ double phase functional can be approximated by some more regular variants. 
We regularize the functional by adding $\epsilon$ either to the exponent of the 
first term (so that the problem is in $W^{1,1+\epsilon}(\Omega)$) or to the weight 
$a$ (in which case the problem is in $W^{1,2}(\Omega)$). For brevity, we present the 
proof only for one case which includes both these regularizations: 
\[
\I_\epsilon(u,A):=  \int_A |\nabla u|^{1+\epsilon} + (\epsilon + a(x)^2) |\nabla u|^2 + |u-f|^2  \, dx.
\]
We start with a lower bound for $\I$, which is the more difficult part. 

\begin{lem}\label{lem:limsup}
Let $F\subset \Omega$ be closed and $a\in C^{0,1}(\Omega)$. 
For $\epsilon_i\to 0^+$ and $u\in BV^{1,2}_a(\Omega)$, 
there exist $u_i\in W^{1,2}(U)$ in a neighborhood $U$ of $F$ such that 
\[
\limsup_{i\to \infty} \I_{\epsilon_i}(u_i,F) \le \I(u,F).
\]
\end{lem}
\begin{proof}
Let $u_\delta:= u * \eta_\delta$ be the convolution with the standard mollifier and 
assume that $\delta<\operatorname{dist}(F,\partial \Omega)$. 
By \cite[Lemma~4.5]{HarHL08} and classical $L^2$-results 
\[
\limsup_{\delta\to 0} \Big[ |Du_\delta|(F) +\int_F |u_\delta-f|^2\, dx \Big]
\le
|Du|(F) + \int_F |u-f|^2 \, dx.
\]

For the term with the weight $a$, we consider two cases and use the different expressions 
from \eqref{eq:convolution}. If $0<a(x)\le 2 a(y)$ for all $y\in B(x,\delta)$, then 
\begin{equation*}
a(x) |\nabla u_\delta| 
\le
2\int_\Rn a(y) |\nabla u(y)|\,\eta_\delta(x-y)\, dy
\lesssim 
M(a|\nabla u|)(x);
\end{equation*}
note that the condition $0<a(y)$ with $u\in W^{1,2}_a(\Omega)$ 
ensures that $Du=\nabla u$ is absolutely continuous 
in $B(x,\delta)$ and note also that the last inequality follows from elementary estimates 
(e.g.\ \cite[Lemma~4.6.3]{DieHHR11}). Furthermore, since $a|\nabla u|\in L^2(\Omega)$ and 
the maximal operator is bounded on $L^2(\Omega)$, we see that the function 
on the right-hand side is in $L^2(\Omega)$, as well. If $a(x)=0$, then the estimate trivially holds. 
Suppose then that $a(x)> 2 a(y)$ for some $y\in B(x,\delta)$. Since $a\in C^{0,1}(\Omega)$, 
we obtain the inequality
\[
a(x) - c\,|x-y| \le a(y) \le \tfrac12 a(x),
\]
so that $a(x) \lesssim |x-y|\le \delta$. Therefore
\begin{equation}\label{eq:conv2}
a(x) |\nabla u_\delta| 
\lesssim
\int_\Rn \delta \,|u(y)\nabla\eta_\delta(x-y)|\, dy
\lesssim
\frac1{|B(x,\delta)|}\int_{B(x,\delta)} |u(y)|\, dy
\le Mu(x),
\end{equation}
where we used that $|\delta \nabla \eta_\delta|\lesssim  \delta^{-n} \chi_{B(x,\delta)}$ 
for the middle step. 
Again, since $u\in L^2(\Omega)$, we obtain an upper bound independent of $\delta$ 
in the space $L^2(\Omega)$. In the set $\{a>0\}$ we have $\nabla u_\delta\to \nabla u$ 
almost everywhere. Thus it follows by dominated convergence in $L^2(\Omega)$ that 
\[
\lim_{\delta\to 0} \int_F (a(x) |\nabla u_\delta|)^2 dx 
=
\int_F (a(x) |\nabla u|)^2 dx.
\]

We have so far shown that 
\[
\limsup_{\delta\to 0} \I(u_\delta, F) \le \I(u,F). 
\]
It remains to change the first functional from $\I$ to $\I_{\epsilon_i}$. 
Equation~\eqref{eq:convolution} implies $|\nabla u_\delta| \le \frac c {\delta^n}$, where $c$ depends on $|Du|(\Omega)$. Therefore 
\[
\int_F |\nabla u_\delta|^{1+\epsilon} + (\epsilon + a(x)^2) |\nabla u_\delta|^2 \, dx
\le 
(\tfrac c {\delta^n})^\epsilon |Du_\delta|(F) 
+ \int_F (a(x) |\nabla u_\delta|)^2  \, dx + \epsilon(\tfrac c {\delta^n})^2 |F|. 
\]
We choose $\delta_i:=\epsilon_i^{1/(3n)}$ so that $(\tfrac c {\delta_i^n})^{\epsilon_i}\to 1$ and 
$\epsilon_i(\tfrac c {\delta_i^n})^2 \to 0$ and set $u_i:=u_{\delta_i}$. Then 
\begin{align*}
&\limsup_{i\to \infty} \I_{\epsilon_i}(u_i, F) \\
&\qquad \le
\limsup_{i\to \infty} \Big[ (\tfrac c {\delta_i^n})^{\epsilon_i} |Du_i|(F) 
+ \int_F (a(x) |\nabla u_i|)^2 +|u_i-f|^2 \, dx 
+ \epsilon_i(\tfrac c {\delta_i^n})^2 |F| \Big]\\
&\qquad =
\limsup_{i\to \infty} \Big[ |Du_i|(F) + \int_F (a(x) |\nabla u_i|)^2 +|u_i-f|^2 \, dx \Big]
\le \I(u,F). \qedhere
\end{align*}
\end{proof}

\begin{remark}\label{rem:condition}
From the previous proof we can see that the exact condition used 
for $a$ is not $C^{0,1}(\Omega)$, but rather the inequality
\[
a(x) \lesssim \max\{|x-y|, a(y)\}
\quad\text{for all }
x,y\in \Omega.
\]
This means that we could replace $a(x)^2$ in the double phase functional with $a(x)^q$ for 
$a\in C^{0,\alpha}(\Omega)$ as long as $q\alpha \ge 2$. 
This kind of condition was first 
identified for the double phase functional in \cite[Section~7.2]{HarH19}. 
\end{remark}

With the method of the previous proof, one can obtain from \eqref{eq:conv2}
that $a(x)|\nabla u_\delta|$ is bounded 
by $M_{\alpha}(Du)$ when $a\in C^{0,\alpha}(\Omega)$ and $M_\alpha$ denotes the 
fractional maximal operator (cf.\ Lemma~\ref{lem:limsup2}). 
This will allow us to prove the result for bounded functions $u$ with a larger 
class of weights $a$. 
A number of recent studies, e.g.\ \cite{CarM17, CarMP17}, deal with the question of
the Sobolev regularity of the maximal function $M_\alpha u$ of a Sobolev or $BV$ function $u$.
However, we have not found any results on the maximal function of the 
derivative of a $BV$ function. Therefore, the following result may be of independent interest.

\begin{prop}\label{prop:fractional}
Let $\mu$ be a vector Borel measure in $\Omega$ with finite total variation $|\mu|(\Omega)<\infty$, 
$\sigma\in (0,n)$ and $\alpha\in (0,n-\sigma)$.
Then the capped fractional maximal function 
\[
M_\alpha^\sigma \mu(x) := 
\sup_{r\le\diam\Omega} \frac{\min\{|\mu|(B(x,r)),r^\sigma\}}{|B(x,r)|^{1-\frac\alpha n}}
\]
belongs to $L^p(\Omega)$ if $p< 1 +\frac\alpha{n-\sigma-\alpha}$. 

Furthermore, the bound is sharp since the claim does not hold 
for $p \ge 1 +\frac\alpha{n-\sigma-\alpha}$. 
\end{prop}
\begin{proof}
We consider dyadic cubes intersecting $\Omega$ with side-length at most $\diam\Omega$. 
Specifically, we assume that the cubes are of the form $[a_1,b_1)\times\cdots\times [a_n,b_n)$ 
and denote by $\mathcal D_k$ the set of such cubes with side-length $2^k$.
Let $D^x_k\in \D_k$ be the cube which contains $x$ and 
$3D^x_k$ be its threefold dilate. We define $\mu_k(A):=\min\{|\mu|(A\cap\Omega),2^{\sigma k}\}$. 
If $2^{k-1}\le r< 2^k$, then $B(x,r)\subset 3D^x_k$. Thus 
\[
M_\alpha^\sigma \mu(x) 
\lesssim 
\sup_{k\in K_0} \frac{\mu_k(3D^x_k)}{2^{(n-\alpha)k}},
\]
where $K_0:=\{-\infty,\ldots, k_0\}$ and $k_0$ is the smallest integer with $2^{k_0} > \diam\Omega$. 
We raise this to the power $p$ and estimate the supremum by a sum: 
\[
M_\alpha^\sigma \mu(x)^p 
\lesssim 
\sup_{k\in K_0} \Big(\frac{\mu_k(3D^x_k)}{2^{(n-\alpha)k}}\Big)^p
\le
\sum_{k\in K_0} \Big(\frac{\mu_k(3D^x_k)}{2^{(n-\alpha)k}}\Big)^p. 
\]
Next we integrate over $\Omega$ and use that $\mu_k(3D^x_k)$ can be estimated 
by the sum of $3^n$ terms of the form $\mu_k(D_k)$ with $D_k\in \D_k$. 
Thus we obtain that 
\begin{align*}
\int_\Omega M_\alpha^\sigma \mu(x)^p\, dx
& \lesssim 
\sum_{k\in K_0} 2^{ - (n-\alpha)p k} \int_\Omega \mu_k(D^x_k)^p\, dx \\
&\le
\sum_{k\in K_0} 2^{ - (n-\alpha)p k} \sum_{D\in \D_k} \mu_k(D)^p |D|.
\end{align*}
Let us maximize the sum $\sum_{D\in \D_k} \mu_k(D)^p$ separately for each $k$. 
Since $\D_k\cap \Omega$ is a partition of $\Omega$, we can write this optimization 
problem as 
\[
S_k:= \sup\bigg\{ \sum_i a_i^p\, \Big|\, \sum_i a_i\le |\mu|(\Omega),\ 
a_i \in \big[0, 2^{\sigma k}\big]\bigg\} 
\]
where $a_i=\mu_k(D_i)$ for $D_i\in \D_k$; 
the last restriction holds since $\mu_k(\Omega) \le 2^{\sigma k}$ by the definition of 
$\mu_k$. We consider what values of the $a_i$'s leads to a maximally large sum. 
If $0<a_i<a_j< 2^{\sigma k}$, then 
\[
a_i^p+a_j^p < (a_i-t)^p+(a_j+t)^p
\]
for $0<t<\min\{a_i,2^{\sigma k}-a_j\}$. Therefore the sum is maximized 
subject to the constraints when $a_i=2^{\sigma k}$ for as many indices as possible and 
zero for the rest.
There are no more than $\lceil 2^{-\sigma k}|\mu|(\Omega) \rceil$ such maximal indices. Thus 
\[
S_k \approx 2^{-\sigma k}\,|\mu|(\Omega)\, 2^{\sigma k p}\, \approx 2^{(p-1)\sigma k}. 
\]
We use this estimate in our previous inequality, and conclude that 
\[
\int_\Omega M_\alpha^\sigma \mu(x)^p \, dx
\lesssim 
\sum_{k\in K_0} 2^{- (n-\alpha)p k} 2^{(p-1)\sigma k} 2^{nk}
=
\sum_{k\in K_0} 2^{[ - (n-\alpha)p + (p-1)\sigma + n ]k}.
\]
The last sum is finite if $-(n-\alpha)p +(p-1)\sigma + n>0$, 
which is equivalent to the condition in the proposition.

It remains to prove sharpness. For simplicity we consider only the case when $\sigma$ 
is an integer. We let $E$ be a $\sigma$-dimensional plane and define 
$\mu(A):=\H^{\sigma}(E\cap A)$. Denote $d(x):=\dist(x,E)$. Then 
\[
M_\alpha^\sigma \mu(x) 
\gtrsim
\frac{\mu(B(x,2d(x)))}{|B(x,2d(x))|^{1-\frac\alpha n}} 
\approx 
d(x)^{\sigma-n+\alpha}.
\]
We raise this to the power $p$ and integrate over $x$:
\[
\int_\Omega M_\alpha^\sigma \mu(x)^p \,dx 
\gtrsim
\int_\Omega d(x)^{(\sigma-n+\alpha)p} dx 
\approx 
\int_0^1 r^{(\sigma-n+\alpha)p} r^{n-\sigma-1} dr. 
\]
This integral diverges if $(\sigma-n+\alpha)p + n-\sigma \le 0$, 
which gives the claimed bound for $p$. 
In the case of non-integer $\sigma$, we instead choose our set as 
the Cartesian product of a plane and a Cantor set, and estimate as before. 
\end{proof}

With the fractional maximal operator we can extend Lemma~\ref{lem:limsup} in the case 
of bounded functions. Bounded functions are very natural in the context of 
image processing, since the grey-scale values are usually taken in some compact interval 
such as $[0,255]$ or $[0,1]$. 
Note that to use the previous proposition, we cannot directly move to 
the total variation measure $|Du|$, since this is not in general going to 
satisfy the appropriate decay $r^{n-1}$ when $u$ is bounded. 
Rather, we have to first estimate the absolute value of 
the measure of a ball, $|Du(B(x,r))|$, and only afterward move to $|Du|$. 
In the next result we therefore work with the vector measure 
$Du$ rather than its total variation, which makes the estimates slightly more difficult. 

\begin{lem}\label{lem:limsup2}
Let $F\subset \Omega$ be closed and $a\in C^{0,\alpha}(\Omega)$ for some $\alpha>\frac12$. 
For $\epsilon_i\to 0^+$ and $u\in BV^{1,2}_a(\Omega)\cap L^\infty(\Omega)$, 
there exist $u_i\in W^{1,2}(U)\cap L^\infty(\Omega)$ in a neighborhood $U$ of $F$ such that 
\[
\limsup_{i\to \infty} \I_{\epsilon_i}(u_i,F) \le \I(u,F).
\]
\end{lem}
\begin{proof}
The proof is identical to that of Lemma~\ref{lem:limsup}, except for the 
estimate of $a(x) |\nabla u_\delta|$ in the second case, $a(x)< \frac12 a(y)$. 
Let us show that we can use Proposition~\ref{prop:fractional} to handle this case. 
By the construction of the measure $Du$, 
\[
\int_{B(x,r)} \phi\cdot dDu
= 
- \int_{B(x,r)} u \div \phi\, dy
\]
for all $\phi\in C^1_0(B(x,r);\Rn)$, cf.\ \cite[Proposition~3.6]{AmbFP00}. We choose 
$\phi(y) = b \xi(|x-y|)$ where $b\in B(0,1)$ and $\xi\in C^1([0,\infty))$ with 
$\xi|_{[0, r-\epsilon-\epsilon^2]}=1$, $\xi|_{[r-\epsilon^2,\infty)}=0$ and $|\xi'| \le \frac2\epsilon$. Then 
$|\div \phi| \le \frac2\epsilon \chi_{B(x,r-\epsilon^2)\setminus B(x,r-\epsilon-\epsilon^2)}$ 
and so
\[
\Big| \int_{B(x,r)} u \div \phi\, dy\Big|
\le 
\|u\|_\infty \tfrac2\epsilon \, \big|B(x,r-\epsilon^2)\setminus B(x,r-\epsilon-\epsilon^2)\big|
\approx 
r^{n-1}
\]
since $u$ is bounded. It follows by monotone convergence as $\epsilon\to 0^+$ that 
\[
|Du(B(x,r))|
= 
\sup_{|b|=1} b\cdot Du(B(x,r)) 
\lesssim 
r^{n-1}.
\]
Therefore, $|Du(B(x,r))| \lesssim \min\{|Du|(B(x,r)), r^{n-1}\}$ and so 
\[
|Du(B(x,r))| \lesssim M_\alpha^{n-1}(Du)(x) r^{n-\alpha}.
\]

On the other hand, we can estimate for the derivative of the convolution using \eqref{eq:convolution}, 
the distribution function of $Du$ \cite[Theorem~8.16]{Rud87}
and the estimate $| \frac d{dr} \eta_\delta(re_1)| \lesssim \delta^{-n-1}$. 
For a unit vector $e_1$, it follows that  
\begin{align*}
|\nabla u_\delta| 
&\le
\Big|\int_\Rn \eta_\delta(x-y)\, dDu(y) \Big|
=
\Big|\int_0^\delta \tfrac d{dr} \eta_\delta(re_1)\, Du(B(x,r))\, dr \Big| \\
&\le 
\int_0^\delta \big| \tfrac d{dr} \eta_\delta(re_1)\big| \, M_\alpha^{n-1}(Du)(x)r^{n-\alpha}\, dr \\
&\lesssim 
M_\alpha^{n-1}(Du)(x) \delta^{-n-1} \int_0^\delta r^{n-\alpha}\, dr
\approx
\delta^{-\alpha} M_\alpha^{n-1}(Du)(x).
\end{align*}
As in Lemma~\ref{lem:limsup}, we conclude now from $a\in C^{0,\alpha}(\Omega)$ in the second case that 
$a(x) \le \delta^{\alpha}$. Thus  $a(x) |\nabla u_\delta| \lesssim M_\alpha^{n-1}(Du)(x)$.
By Proposition~\ref{prop:fractional}, the right-hand side is in $L^2(\Omega)$ 
provided $2 < 1 + \frac\alpha{n - (n-1) -\alpha} = \frac1{1-\alpha}$, 
which holds since $\alpha>\frac12$. Thus we can use this as the bound 
for dominated convergence. The rest of the proof is as before. 
\end{proof}

\begin{remark}
If we consider a double phase functional $t^p + a(x)t^q$ in ``normal'' form, then 
the condition from the previous results can be written $q<p+\alpha$. 
This condition has proved to be of central importance when considering 
bounded solutions, cf.\ \cite{BarCM18, ColM15b, HarHL_pp18}. 
In this sense, the assumption in Lemma~\ref{lem:limsup2} is probably 
essentially sharp.

However, more precise research has established that one may even 
take $q\le p+\alpha$ for bounded minimizers \cite{BarCM18, HarHT17} (see also \cite{DeFM_pp1, HasO_pp19} for 
the borderline case with unbounded minimizers). 
The borderline is handled using additional H\"older continuity obtained via De Giorgi
technique, which in this case implies that $u\in C^{0,\gamma}(\Omega)$ for some $\gamma>0$. 
Indeed, from the previous proof we can see that 
$a\in C^{0,1/2}(\Omega)$ would suffice if we had $u\in C^{0,\gamma}(\Omega)$ 
for some positive $\gamma>0$ (as one has when $p,q>1$) instead of $u\in L^\infty(\Omega)$. 
However, for $BV$ problems, such higher regularity of the function cannot be expected. 
Therefore, the borderline $q=p+\alpha$ remains a problem for future research. 

Let us also note that Ok \cite{Ok_pp} has considered double phase functionals 
under additional a priori integrability assumptions other than $L^\infty(\Omega)$. 
If one could prove decay estimates $|Du(B(x,r))|\lesssim r^\sigma$ for $\sigma\in (n-1,n)$
when $u\in L^s(\Omega)$, we could cover also this case. 
We do not know about such of results, so this, likewise, remains for a topic for another study. 
\end{remark}


\section{Upper estimates for the \texorpdfstring{$BV$}{BV} double phase functional}

The concept of $\Gamma$-convergence, introduced by De Giorgi and Franzoni \cite{DeGF75}, has been systematically 
presented in \cite{Bra02, Dal93}. 
A family of functionals $\I_\epsilon: X \to \overline{\R}$ is said to \textit{$\Gamma$-converge}
(in topology $\tau$) to $\I: X \to \overline{\R}$
if the following hold for every positive sequence $(\epsilon_i)$ converging to zero:
\begin{enumerate}
 \item[(a)]  $\displaystyle \I (u) \le \liminf_{i \to \infty} \I_{\epsilon_i} (u_{i})$ 
for every $u \in X$ and every $(u_{i})\subset X$ $\tau$-converging to $u$;
 \item[(b)]
$\displaystyle \I (u) \ge \limsup_{i \to \infty} \I_{\epsilon_i} (u_{i})$
for every $u \in X$ and some $(u_{i})\subset X$ $\tau$-converging to $u$.
\end{enumerate}

Let us remark that the somewhat strange assumption $\H^{n-1}(\{a=0\}\cap \partial\Omega)=0$
in the next theorem is actually 
quite natural: since $\{a=0\}$ is the set where the image edges occur, we cannot identify 
the edge if it coincides with the image boundary $\partial \Omega$. On the other hand, 
we also have no need for the jump in the function at this location, since the other part 
of the jump will be outside the image, and thus cannot be seen.

\begin{thm}\label{thm:Gamma}
Suppose that $\Omega$ is a rectangular cuboid, $a\in C^{0,1}(\overline \Omega)$, 
and assume that $a>0$ $\H^{n-1}$-a.e.\ on the boundary $\partial \Omega$. 
Then $\I_\epsilon$ $\Gamma$-converges to  $\I$ in $L^1(\Omega)$ topology with 
$X:=BV^{1,2}_a(\Omega)$.
\end{thm}
\begin{proof}
Let us start with condition (a) in the definition of $\Gamma$-convergence.
Let $(\epsilon_i)$ be a positive sequence converging to zero.
Let $u \in BV^{1,2}_a(\Omega)$ and let $(u_i)\subset BV_a^{1,2}(\Omega)$ 
be a sequence converging to $u$ in $L^1(\Omega)$. 
If  $\liminf_{i \to \infty} \I_{\epsilon_i} (u_i) =\infty$, then there is nothing to prove, 
so we assume that $K:=\liminf_{i \to \infty} \I_{\epsilon_i} (u_i)<\infty$. 
We restrict our attention to a subsequence with $\lim_{i \to \infty} \I_{\epsilon_i} (u_i)=K$ 
and $u_i\in W^{1,2}(\Omega)$. 
Then $(u_i)$ is a bounded sequence in $BV^{1,2}_a(\Omega)$. 
By precompactness of $BV$ there exists a limit function for a subsequence such that 
$|Du^b|(\Omega)\le \liminf |Du_i|(\Omega)$;
by reflexivity of $W^{1,2}_a(\Omega)$ and $L^2(\Omega)$, we obtain subsequences with 
$\nabla u_i \rightharpoonup \nabla  u^w$, $u_i \rightharpoonup u^w$ in $L^2_a(\Omega)$
and $u_i-f \rightharpoonup u^l-f\ \text{in }L^2(\Omega)$.
By $u_i\to u$ in $L^1(\Omega)$ and the uniqueness of the limit, we conclude that 
$u^b=u^w=u^l=u$. 

The weak lower semi-continuity of the Lebesgue integral yields that
\[
\int_\Omega |u-f|^2 \, dx \le \liminf_{i \to \infty} \int_\Omega |u_i-f|^2 \, dx
\] 
and, since $\epsilon_i\ge 0$, 
\[
\int_\Omega (a(x) |\nabla u|)^2  \, dx 
\le  \liminf_{i \to \infty} \int_\Omega (a(x)|\nabla u_i|)^2 \, dx
\le \liminf_{i \to \infty} \int_\Omega (\epsilon_i + a(x)^2)|\nabla u_i|^2 \, dx.
\] 
Finally, for the $BV$ part we use the estimate from the previous paragraph, Young's inequality 
and $(\frac{1}{1+ \epsilon_i})^{1/{\epsilon_i}} \to \frac1e$: 
\[
\begin{split}
|Du|(\Omega) 
&\le \liminf_{i \to \infty}|Du_i|(\Omega) 
= \liminf_{i\to\infty} \int_\Omega |\nabla u_i|\, dx \\
&\le \liminf_{i \to \infty} \int_\Omega |\nabla u_i|^{1+\epsilon_i} + 
\Big(\frac{1}{1+ \epsilon_i}\Big)^{\frac1{\epsilon_i}} \frac{\epsilon_i}{1-\epsilon_i} \, dx
= \liminf_{i \to \infty} \int_\Omega |\nabla u_i|^{1+\epsilon_i} \, dx.
\end{split}
\] 
By combining the above inequalities we obtain condition (a). Note that 
for this part we do not need the assumptions on $\Omega$ and $a$.

Let us then move to condition (b). Since $\Omega$ is a rectangular cuboid, 
we can extend both the function $u$ and the weight $a$ by reflections to 
the rectangular cuboid with the same center but $3$ times the side-lengths. 
Then we use Lemma~\ref{lem:limsup} with $F:=\overline\Omega$ to conclude that 
there exist $u_i\in W^{1,2}(U)$ such that
\[
\limsup_{i\to \infty} \I_{\epsilon_i}(u_i,\overline\Omega) \le \I(u,\overline\Omega).
\]
We need this inequality with $\Omega$ instead of $\overline\Omega$. 
Since $|\partial \Omega|=0$ and $u_i$ is a Sobolev function, 
$\I_{\epsilon_i}(u_i,\overline\Omega)=\I_{\epsilon_i}(u_i,\Omega)$. 
On the right-hand side, the same reason implies that
\[
\int_{\overline \Omega} (a(x)|\nabla u|)^2 + |u-f|^2\, dx 
=
\int_{\Omega} (a(x)|\nabla u|)^2 + |u-f|^2\, dx.
\]
The singular set of $Du$ is contained in $\{a=0\}$ because $u\in W^{1,2}_a(U)$. 
Since $\{a=0\}\cap \partial \Omega$ 
has Hausdorff $(n-1)$-measure zero by assumption, it follows by the decomposition 
\eqref{eq:decomposition} that $|Du|(\partial \Omega)=0$ and 
so $|Du|(\overline \Omega) = |Du|(\Omega)$. Thus we 
have established condition (b) of $\Gamma$-convergence. 
\end{proof}

In the previous theorem we could consider a Lipschitz domain instead of 
a rectangular cuboid. In this case, the extension of both $u$ and $a$ would 
be done by flattening the boundary with the Lipschitz map. 
If we use Lemma~\ref{lem:limsup2} instead of Lemma~\ref{lem:limsup}, we 
obtain the following variant. 

\begin{thm}\label{thm:Gamma2}
Suppose that $\Omega$ is a bounded Lipschitz domain, $a\in C^{0,\alpha}(\overline \Omega)$ 
for some $\alpha>\frac12$, and assume that $a>0$ $\H^{n-1}$-a.e.\ on the boundary $\partial \Omega$. 
Then $\I_\epsilon$ $\Gamma$-converges to $\I$ in $L^1(\Omega)$ topology with 
$X:=BV^{1,2}_a(\Omega)\cap L^\infty(\Omega)$.
\end{thm}

We use the following formulation for relaxation, which emphasizes the connection with 
$\Gamma$-convergence. A functional $\overline\J : X \to \overline\R$ is
the \textit{relaxation of $\J:X\to\overline\R$} in topology $\tau$ if 
\begin{enumerate}
 \item[(a)] $\displaystyle\overline{\J} (u) \le \liminf_{i \to \infty} \J (u_{i})$ 
for every $u \in X$ and every $(u_{i})\subset X$ $\tau$-converging to $u$;
 \item[(b)] $\displaystyle\overline\J (u) \ge \limsup_{i \to \infty} \J (u_{i})$ 
for every $u \in X$ and some $(u_{i})\subset X$ $\tau$-converging to $u$.
\end{enumerate}
The relaxation is the greatest lower-semicontinuous minorant of $\J$. See \cite[Proposition~1.31, p.~33]{Bra02}.
Let us write for $u \in BV(\Omega)$ that
\[
\J(u):= 
\begin{cases}
\int_\Omega  |\nabla u| + (a(x) |\nabla u|)^2 + |u-f|^2  \, dx, & \text{if } u\in W^{1,1}(\Omega) \\
\infty, & \text{if } u\in BV(\Omega)\setminus W^{1,1}(\Omega).
\end{cases} 
\]
We show that the relaxation $\overline \J$ of this functional equals $\I$. 
The proof is identical to Theorem~\ref{thm:Gamma}, we simply take $\I_\epsilon=\J$ for 
every $\epsilon>0$ and $\I=\overline\J$. Naturally, we could also 
prove an analogue to Theorem~\ref{thm:Gamma2}. 

\begin{cor}\label{cor:relaxation}
Suppose that $\Omega$ is a rectangular cuboid, $a\in C^{0,1}(\overline \Omega)$, 
and assume that $a>0$ $\H^{n-1}$-a.e.\ on the boundary $\partial \Omega$. 
Then $\overline\J = \I$ in $L^1(\Omega)$ topology.
\end{cor}

\section*{Acknowledgement}

We thank the referee for some comments regarding this manuscript.


\bibliographystyle{amsplain}

\end{document}